\documentclass[preprint]{imsart}

\usepackage{amsthm,amsmath,amssymb,enumerate,xcolor}

\startlocaldefs
\newcommand*{\abs}[1]{\left|#1\right|}

\newcommand{\Prob}{\mathbf P}
\renewcommand{\P}{\mathbf P}
\newcommand{\E}{\mathbf E}
\newcommand{\R}{\mathbb R}

\newcommand{\ind}{\mathbf 1}

\newtheorem{theorem}{Theorem}[section]
\newtheorem{lemma}[theorem]{Lemma}
\newtheorem{proposition}[theorem]{Proposition}
\newtheorem{corollary}[theorem]{Corollary}

\newtheorem{rem}{Remark}

\endlocaldefs

\begin{document}

\begin{frontmatter}

\title{Consistency of the drift  parameter estimator for the discretized fractional  Ornstein--Uhlenbeck process  with Hurst index  $H\in(0,\frac12)$}
\runtitle{Consistency of the drift  parameter estimator}


\author{\fnms{K\c{e}stutis} \snm{Kubilius}\thanksref{t1}\ead[label=e1]{kestutis.kubilius@mii.vu.lt}}
\address{Vilnius University, Institute of Mathematics and Informatics,\\
Akademijos 4, LT-08663, Vilnius, Lithuania\\
\printead{e1}}
\author{\fnms{Yuliya} \snm{Mishura}\thanksref{t1}\corref{}\ead[label=e2]{myus@univ.kiev.ua}}
\thankstext{t1}{This research was funded by a grant (No. VIZ-TYR-110) from the Research Council of Lithuania and  by a grant  No. Dnr 300-1508-09   of Ume{\aa} University, Sweden}
\address{Department of Probability Theory, Statistics and Actuarial Mathematics,\\
Taras Shevchenko National University of Kyiv,\\
64 Volodymyrska, 01601 Kyiv, Ukraine\\
\printead{e2}}
\author{\fnms{Kostiantyn} \snm{Ralchenko}\ead[label=e3]{k.ralchenko@gmail.com}}
\address{Department of Probability Theory, Statistics and Actuarial Mathematics,\\
Taras Shevchenko National University of Kyiv,\\
64 Volodymyrska, 01601 Kyiv, Ukraine\\
\printead{e3}}
\author{\fnms{Oleg} \snm{Seleznjev}\thanksref{t1}\ead[label=e4]{oleg.seleznjev@math.umu.se}}
\address{Institute of Mathematics and Mathematical Statistics,\\
University of Ume\aa,\\
SE-90187 Ume\aa, Sweden\\
\printead{e4}}

\runauthor{K.~Kubilius et al.}

\begin{abstract}
We consider Langevin equation  involving fractional Brownian motion with Hurst index $H\in(0,\frac12)$. Its solution is the fractional Ornstein--Uhlenbeck process and with unknown drift parameter $\theta$. We construct the estimator that is similar in form to maximum likelihood estimator for Langevin equation with standard Brownian motion. Observations are discrete in time. It is assumed that the interval between observations is $n^{-1}$, i.e.  tends to zero (high frequency data) and the number of observations increases to infinity as  $n^m$ with $m>1$. It is proved that for positive $\theta$ the estimator is strongly consistent for any $m>1$ and for negative $\theta$ it is consistent when $m>\frac{1}{2H}$.
\end{abstract}





\end{frontmatter}

\section{Introduction}
Let $(\Omega, \mathfrak{F}, \P)$ be a complete probability space. We consider fractional Brownian motion $B^H=\{B^H_t,\;t\geq 0\}$ on this probability space, that is, the centered Gaussian process with the covariance function
$$R(t,s)=\frac12\left(s^{2H}+t^{2H}-|t-s|^{2H}\right).$$ We restrict ourselves to the case $H\in\left(0,\frac12\right)$ and consider the continuous (and even H\"{o}lder up to order $H$) modification that exists due to the Kolmogorov theorem.
Introduce  the Langevin equation,\begin{equation}\label{equation}
X_t=x_0+\theta\int_0^tX_s\,ds+B^H_t,
\quad t\ge0,
\quad H\in(0,\tfrac12).
\end{equation}
According to Proposition A.1 from \cite{CKM}, this equation has the unique solution that is named fractional Ornstein--Uhlenbeck process and can be presented as
\begin{equation}\label{solution}
X_t=x_0e^{\theta t}+\theta e^{\theta t}\int_0^te^{-\theta s}B^H_s\,ds +B^H_t, \;t\geq 0.
\end{equation}

The goal of the paper is to construct   consistent (strongly consistent) estimator of the unknown drift parameter $\theta$ by discrete observations of the process $X$.

The problem of the estimation of the drift parameter $\theta$ in the linear equation containing fBm and in the equation \eqref{equation} when the Hurst index $H\ge\frac12$ was investigated in many works. For linear models, mention only papers \cite{bertin} and \cite{HuXZ}.
Drift parameter estimators for fractional Ornstein--Uhlenbeck process with continuous time when the whole trajectory of $X$ is observed, were studied in \cite{BESO,HN,KB}.
Kleptsyna and Le Breton \cite{KB} constructed the maximum likelihood estimator and proved its strong consistency for any $\theta\in\R$. They also investigated the asymptotic behaviour of the bias and mean square error of this estimator.
The sequential maximum likelihood estimation was considered in \cite{Prakasa04}.
Hu and Nualart \cite{HN} proved that in the ergodic case ($\theta<0$) the least square estimator
\begin{equation}\label{lse_cont}
\widehat\theta_T=\frac{\int_0^TX_t\,dX_t}{\int_0^TX_t^2\,dt},
\end{equation}
is strongly consistent for all $H\ge\frac12$ and asymptotically normal for $H\in[\frac12,\frac34)$.
They also obtained the strong consistency and asymptotic normality of the estimator
\begin{equation}\label{HN_cont}
\widehat\theta_T=\left(\frac{1}{H\Gamma(H)T}\int_0^TX_t^2\,dt\right)^{-\frac{1}{2H}}.
\end{equation}
In \cite{BESO} the corresponding non-ergodic case $\theta>0$ was considered and the strong consistency of the least square estimator \eqref{lse_cont} was proved for $H>\frac12$.
It was obtained also that
$e^{\theta t}\left(\widehat\theta_t-\theta\right)$
converges in law to $2\theta\mathcal C(1)$ as $t\to\infty$, where $\mathcal C(1)$ is the standard Cauchy distribution.
Minimum contrast estimators in continuous and discrete case were studied in \cite{Bishwal11}.
The distributional properties of maximum likelihood, minimum contrast and least square estimators were explored in \cite{Tanaka13}.
For the two-parameter generalization see \cite{CT}.

In \cite{CES,Es-Sebaiy,ESN} the discretized version of \eqref{lse_cont} is considered, namely
\begin{equation}\label{lse_dis}
\widehat\theta_n=\frac{\sum_{i=1}^{n}X_{t_{i-1}}\left(X_{t_i}-X_{t_{i-1}}\right)}%
{\Delta_n\sum_{i=1}^{n}X_{t_{i-1}}^2},
\end{equation}
where the process $X$ was observed in the points $t_i = i\Delta_n$, $i=0,\dots,n$, such that
$\Delta_n\to0$ and $n\Delta_n\to\infty$ as $n\to\infty$.
In \cite{CES} the ergodic case $\theta<0$ was studied, the strong consistency of this estimator was proved for $H\ge\frac12$ and the almost sure central limit theorem was obtained for $H\in\left(\frac12,\frac34\right)$.
The non-ergodic case $\theta>0$ was considered in Es-Sebaiy and Ndiaye \cite{ESN}.
They proved the strong consistency of the estimator \eqref{lse_dis} for $H\in\left(\frac12,1\right)$ assuming that
$\Delta_n\to0$ and $n\Delta_n^{1+\alpha}\to\infty$ as $n\to\infty$ for some $\alpha>0$. The same result was obtained for the estimator
\[
\widehat\theta_n=\frac{X_{t_n}^2}
{2\Delta_n\sum_{i=1}^{n}X_{t_{i-1}}^2}.
\]

In \cite{HuSong13,xiaoZX} the following discretized version of the estimator \eqref{HN_cont} was considered
\[
\widehat\theta_n=-\left(\frac{1}{nH\Gamma(2H)}\sum_{k=1}^nX_{k\Delta}^2\right)^{-\frac{1}{2H}},
\]
where $\theta<0$ and the process $X$ was observed in the points $\Delta,2\Delta,\dots,n\Delta$ for some fixed $\Delta>0$.
Hu and Song \cite{HuSong13} proved the strong consistency of the estimator for $H\ge\frac12$ and the asymptotic normality for $\frac12\le H<\frac34$.

In \cite{BI, ZXZN} more general situation is studied, where the equation has the following form
{$dX_t=\theta X_t dt+\sigma dB_t^H$, $t>0$,}
and $\vartheta=(\theta,\sigma,H)$ is the unknown parameter, $\theta<0$.
Consistent and asymptotically Gaussian estimators of the parameter $\theta$ are proposed
by the discrete observations of the sample path
$(X_{k\Delta_n}, k=0,\dots,n)$
for $H\in(\frac12,\frac34)$,
where
$n\Delta_n^p\to\infty$, $p>1$,
and $\Delta_n\to0$ as $n\to\infty$.
In \cite{ZXZN} the strongly consistent estimator is constructed for the scheme, when $H> \frac12$, the time interval $[0,T]$ is fixed  and the process is observed at points $h_n,2h_n,\dots,nh_n$, where $h_n=\frac{T}{n}$.

In \cite{DMM,Mendy} the so called sub-fractional Ornstein--Uhlenbeck process was studied, where the process $B^H_t$ in \eqref{equation} was replaced with a sub-fractional Brownian motion. In \cite{DMM} the maximum likelihood estimator for such process was constructed,
in \cite{Mendy} the estimator \eqref{lse_cont} was investigated in the case $\theta>0$. The maximum likelihood drift parameter estimators  for  fractional Ornstein--Uhlenbeck process and even more   general processes involving fBm with Hurst index
from the whole interval $(0,1)$ were constructed and studied in \cite{tudorviens}. These estimators involve singular kernels therefore are more complicated to study and simulate. To the best of our knowledge, it is the only paper when discretized estimates of the drift parameter are constructed in the case $H<\frac12$. However, the observations of the real financial markets demonstrate that the Hurst index often falls below the level of $\frac12$, taking values around 0.45--0.49 (\cite{multi}). In order to consider the case of $H<\frac12$ and to overcome the technical difficulties connected with  singular kernels, we construct comparatively simple estimator that is similar in form to the maximum  likelihood estimator for Langevin equation with standard Brownian motion. Observations are assumed to be discrete in time and we assume  that the interval between observations is $n^{-1}$, i.e.  tends to zero, so we consider high frequency data. At the same time,  the number of observations increases to infinity with the speed   $n^m$ with $m>1$. Let $n\ge1$, $t_{k,n}=\frac{k}{n}$, $0\le k\le n^m$, where $m\in \mathbb{N}$ be some fixed number.
Suppose that we observe $X$ at the points
$\{t_{k,n},n\ge1,0\le k\le n^m\}$.
Consider the estimator
\begin{equation}\label{main-main}
\widehat{\theta}_n(m)=\frac{\sum_{k=0}^{n^m-1}X_{k,n}\Delta X_{k,n}}%
{\frac1n\sum_{k=0}^{n^m-1}X_{k,n}^2},
\end{equation}
where
$X_{k,n}=X_{t_{k,n}}$,
$\Delta X_{k,n}=X_{k+1,n}-X_{k,n}$.

By \eqref{equation}, estimator $\widehat{\theta}_n(m)$ from \eqref{main-main} can be represented in the following form, which is more convenient for evaluation:
\begin{equation}\label{estimator}
{\widehat{\theta}_n(m)=\theta+\frac{\theta\sum_{k=0}^{n^m-1}X_{k,n}
\int_{\frac{k}{n}}^{\frac{k+1}{n}}\left(X_s-X_{k,n}\right)ds
+\sum_{k=0}^{n^m-1}X_{k,n}\Delta B^H_{k,n}}{\frac1n\sum_{k=0}^{n^m-1}X_{k,n}^2}\,.}
\end{equation}
It is proved that for positive $\theta$ the estimator is strongly consistent for any $m>1$ and for negative $\theta$ it is consistent for $m>\frac{1}{2H}$.

Our paper is organized as follows. In Section \ref{section2} we consider auxiliary result, namely,    bounds with probability 1 for the values and increments of fractional Brownian motion and fractional Ornstein--Uhlenbeck process.  The bounds are factorized to the increasing non-random function and random variable not depending on time. {In Section \ref{section3} we} get the bounds for the numerator of the estimator, while in Section  \ref{section4} we relate discretized integral sum in the denominator of the estimator to the corresponding integral $\int_0^t X_s^2\,ds$. This relation is convenient for some values of parameters because it is easier to apply   L'H\^{o}pital's rule to the integral $\int_0^t X_s^2\,ds$ than Stolz--Ces\`{a}ro theorem to the sum  $\sum_{k=0}^{n^m-1}X_{k,n}^2$  with  terms   depending  on $n$.  {Section \ref{section5} contains} two main theorems, Theorem \ref{main1} and Theorem \ref{main2}  on strong consistency for $\theta>0$ and consistency for $\theta<0$. Section \ref{sectionAp} contains some auxiliary results and Section \ref{section7}  is devoted to numerics.

\section{Bounds for the values and increments of fractional Brownian motion and fractional Ornstein--Uhlenbeck process}\label{section2}

In what follows we shall use next auxiliary estimates for the rate of asymptotic growth with probability 1 of the fractional Brownian motion and its increments. Throughout the paper considering functions of the form $t^p\log t,\; p>0$ we suppose that $0\cdot\infty=0$.
\begin{proposition}
\label{pr:1}
\begin{enumerate}
\item [(i)] For any $p>1$ and any $H\in(0,1)$ there exists nonnegative random variable $\xi(p,H)$ such that
    \begin{equation}\label{est001}
    \sup_{0\le s\le t}\abs{B^H_s}\le\left(\left(t^H|\log t|^p\right)\vee1\right)\xi(p,H),
    \end{equation}
   and there exists such number $c_\xi(p,H)>0$ that for any $0<y<c_\xi(p,H)$,
   $\E\exp\{y\xi^2(p,H)\}<\infty.$

\item [(ii)] For any $q>\frac12$ and any $H\in(0,1)$ there exists nonnegative random variable $\eta(q,H)$ such that for any $0<t_1<t_2<\infty$
   \begin{equation}\label{est002}\abs{B^H_{t_2}-B^H_{t_1}}\leq (t_2-t_1)^H \left(\abs{\log(t_2-t_1)}^{1/2}+1\right) (\log(t_2+2))^q \eta(q,H),
    \end{equation}
    and there exists such number $c_\eta(q,H)>0$ that for any $0<y<c_\eta(q,H)$,
   $\E\exp\{y\eta^2(q,H)\}<\infty.$
\end{enumerate}
\end{proposition}
\begin{proof}
The 1st statement was proved in the paper \cite{KMM}. The 2nd statement follows immediately from the next relation that can be proved similarly to Theorem 1 from \cite {MRSS}, where even more complicated functional than the increment of fractional Brownian motion, more precisely, fractional derivative, was considered. So, we have from Theorem 1, \cite {MRSS}, that for  any $q>\frac12$ and any $H\in(0,1)$  the random variable
    \[
    \eta(q,H)=\sup_{0\le t_1<t_2\le t_1+1} \frac{\abs{B^H_{t_1}-B^H_{t_2}}}{(t_2-t_1)^H
    \left(\abs{\log(t_2-t_1)}^{1/2}+1\right) (\log(t_2+2))^q}
    \]
    is finite almost surely, whence $(ii)$ follows.
\end{proof}
 Now our goal is to estimate the numerator in \eqref{estimator} and compare it to the denominator. At first, we write the bounds for the values of  $X$ and its increments.

\begin{lemma} We have the following bounds for fractional Ornstein-Uhlenbeck process $X$ in terms of supporting fractional Brownian motion:
\begin{itemize}
\item  [$(i)$] Let $\theta>0$. Then for any $t>0$
 \begin{equation}\label{est00}
\sup_{0\leq s\leq t} \abs{X_s}\leq \abs{x_0}e^{\theta t} + \theta e^{\theta t}\int_0^t e^{-\theta s}\sup_{0\leq u\leq s}\abs{B^H_{u}}ds+ \sup_{0\leq s\leq t}\abs{B^H_{s}}
\end{equation} and for any $s\in \left[\frac{k}{n}, \frac{k+1}{n}\right)$
\begin{equation}\label{est11}\begin{gathered}
\sup_{\frac{k}{n}\leq u\leq s}\abs{X_u-X_{k,n}}\le \int_{\frac{k}{n}}^s\left(e^{\theta u}\left(\abs{x_0} +{ {\theta}} \int_0^u e^{-\theta v}\sup_{0\leq z\leq v}\abs{B^H_{z}}dv\right)\right.\\
+\left. \sup_{0\leq z\leq u} \abs{B^H_{z}}\right)du +\sup_{\frac{k}{n}\leq u\leq s}\abs{B^H_{u}-B^H_{k,n}}.
\end{gathered}\end{equation}

\item  [$(ii)$] Let $\theta<0$. Then for any $t>0$
\begin{equation}\label{est0}
\sup_{0\leq s\leq t} |X_s|\leq |x_0| + 2\sup_{0\leq s\leq t} \abs{B^H_{s}}
\end{equation} and for any $s\in \left[\frac{k}{n}, \frac{k+1}{n}\right)$
\begin{equation}\label{est1}\begin{gathered}
\sup_{\frac{k}{n}\leq u\leq s}\abs{X_u-X_{k,n}}\le\frac{\abs{\theta}\abs{x_0}}{n}+\frac{2\abs{\theta}}{n}\sup_{0\leq u\leq s}\abs{B^H_{u}}\\ +\sup_{\frac{k}{n}\leq u\leq s}\abs{B^H_{u}-B^H_{k,n}}.
\end{gathered}\end{equation}
\end{itemize}
\end{lemma}
\begin{proof}$(i)$ Bound \eqref{est00} follows immediately from \eqref{solution},
 and bound \eqref{est11} follows immediately from \eqref{est00} and \eqref{equation}.

 $(ii)$ Bound \eqref{est0} follows   from \eqref{solution}:
\begin{equation*}
\begin{split}
\abs{X_t}&\le\abs{x_0}e^{\theta t}+\abs{\theta}e^{\theta t}
\sup_{0\le s\le t}\abs{B^H_s}\cdot\int_0^te^{-\theta s}ds + \abs{B^H_t}
\leq |x_0| + 2\sup_{0\leq s\leq t} \abs{B^H_{s}}.
\end{split}
\end{equation*}
To establish bound \eqref{est1}, we substitute \eqref{est0} into the following inequality that can be easily obtained  from \eqref{equation}:
 for $s\geq\frac{k}{n}$
\begin{equation*}
\abs{X_s-X_{k,n}}\le\abs{\theta}\int_{\frac{k}{n}}^s\abs{X_u}du
+\abs{B^H_s-B^H_{k,n}}.
\end{equation*}

\end{proof}

\begin{rem} Plugging  $p=2$ and $q=1$ into the formulae \eqref{est001}--\eqref{est002}, we get the following bounds:
\begin{equation}\label{est_KMM}
\sup_{0\le s\le t}\abs{B^H_s}
\le\left(t^H \log^2 t +1\right)\xi(2,H),
\end{equation}
and for $s\in\left[\frac{k}{n}, \frac{k+1}{n}\right]$
\begin{equation}\label{est_MRSS}
\begin{gathered}
\abs{B^H_{s}-B^H_{\frac{k}{n}}}
\le\left(s-\tfrac{k}{n}\right)^H
\left(\abs{\log\left(s-\tfrac{k}{n}\right)}^{1/2}+1\right)\log(s+2)\eta(1,H)\\ \leq \left(\left(s-\tfrac{k}{n}\right)^H\abs{\log\left(s-\tfrac{k}{n}\right)}^{1/2}+\left(s-\tfrac{k}{n}\right)^H\right)\log(s+2)\eta(1,H).
\end{gathered}
\end{equation}
Function $f(x)=x^r\abs{\log x}^{\frac12}$ is bounded on the interval $(0,1]$ for any $r>0$. Therefore $$\left(s-\tfrac{k}{n}\right)^H\abs{\log\left(s-\tfrac{k}{n}\right)}^{1/2}\leq C\left(s-\tfrac{k}{n}\right)^{H-r}$$
for any $0<r<H$. Furthermore, for $s\in\left[\frac{k}{n}, \frac{k+1}{n}\right]$ we have that $\left(s-\tfrac{k}{n}\right)^H\leq \left(s-\tfrac{k}{n}\right)^{H-r}$. Therefore, we get from \eqref{est_MRSS} that for any $0<r<H$ and for $s\in\left[\frac{k}{n}, \frac{k+1}{n}\right]$
\begin{equation}\label{est_MRSS1}
\begin{gathered}
\abs{B^H_{s}-B^H_{\frac{k}{n}}}
\le C \left(s-\tfrac{k}{n}\right)^{H-r}\log(n^{m-1}+2)\eta(1,H).
\end{gathered}
\end{equation}
It follows immediately from  \eqref{est_KMM} that for $\theta>0$
 \begin{equation}\label{lim1}
 \int_0^t e^{-\theta s}\sup_{0\leq u\leq s}\abs{B^H_{u}}ds\leq \xi(2,H)\int_0^t e^{-\theta s}\left(s^H \log^2 s +1\right)ds\leq C\xi(2,H),\end{equation}
 and therefore both   integrals $\int_0^\infty e^{-\theta s}B^H_{s} ds$ and $\int_0^\infty e^{-\theta s}\sup_{0\leq u\leq s}\abs{B^H_{u}}ds$ exist  with probability 1 and admit the same upper bound $C\xi(2,H)$.
Combining \eqref{est00}--\eqref{est1}, \eqref{est_KMM} and \eqref{est_MRSS1}, we get that for $\theta>0$
\begin{equation*}
\sup_{0\leq u\leq s}  |X_u|\leq |x_0| e^{\theta s}+ C\theta e^{\theta s}\xi(2,H)+ \left(s^H\log^2s+1\right)\xi(2,H),
\end{equation*}
and for $s\in\left[\frac{k}{n}, \frac{k+1}{n}\right]$
\begin{equation*}\begin{gathered}
\sup_{\frac{k}{n}\le u\le s}\abs{X_u-X_{k,n}}
\le \theta \int_{\frac{k}{n}}^{s}\left(e^{\theta u}\left(\abs{x_0}+C\theta\xi(2,H)\right)\right.\\
+\left.(u^H\log^2u+1)\xi(2,H)\right)du
+\left(n^{-H+r}\log n\right)\eta(1,H),
\end{gathered}\end{equation*}
while for $\theta<0$
\begin{equation*}
\sup_{0\leq u\leq s}  |X_u|\leq |x_0| + 2\left( s^H\log^2s+1\right)\xi(2,H),
 \end{equation*} and for $s\in\left[\frac{k}{n}, \frac{k+1}{n}\right]$
\begin{equation*}\begin{gathered}
\sup_{\frac{k}{n}\le u\le s}\abs{X_u-X_{k,n}}
\le\frac{\abs{\theta}\abs{x_0}}{n}+\frac{2\abs{\theta}}{n}\sup_{0\leq u\leq s}\abs{B^H_{u}}\\ +\sup_{\frac{k}{n}\leq u\leq s}\abs{B^H_{u}-B^H_{k,n}}\leq \frac{\abs{\theta}\abs{x_0}}{n} +\frac{2\abs{\theta}}{n}
\left(s^H \log^2 s+1\right)\xi(2,H)\\
+\left(n^{-H+r}\log n\right)
\eta(1,H).
\end{gathered}\end{equation*}
To simplify the notations, we denote by $C$ any constant  whose value  is not important for our bounds. Furthermore, we denote by  $\mathfrak{Z}$ the class of  nonnegative random variables with the property: there exists $C>0$ not depending on $n$ such that $\E\exp\{x\zeta^2\}<\infty$ for any $0<x<C$. For example, $\xi(2,H)+C$ and $\eta(1,H)+C$,  $C\xi(2,H)$ and $C\eta(1,H)$ for any constant $C$ belong to  $\mathfrak{Z}$. Also, note that for fixed $m>1$ and $n>3$ we have the upper bound $\log\left(n^{m-1}+3\right)\leq C\log n$. Moreover, for any $\alpha>0$ there exists such $ n(\alpha)$ that for $n\geq  n(\alpha)$ we have $\log n<n^\alpha$. Taking this into account and using the simplified notations,  we get the bounds with the same  $\zeta \in \mathfrak{Z}$:
for $\theta>0$ we have for any fixed $\alpha>0$, starting with $n\geq  n(\alpha)$:
\begin{equation}\begin{gathered}\label{est31}
\sup_{0\leq u\leq s}  |X_u|\leq  \left(e^{\theta s}+s^H\log^2s\right)\zeta
\end{gathered}\end{equation}
and for $s\in\left[\frac{k}{n}, \frac{k+1}{n}\right]$
{\begin{equation}\label{est41}
\sup_{\frac{k}{n}\le u\le s}\abs{X_u-X_{k,n}}
\leq \left(\tfrac{1}{n}{e^{\theta s}} + \tfrac{1}{n}   s^H \log^2 s  +n^{-H+\alpha} \right)\zeta,
\end{equation}}
while for $\theta<0$
\begin{equation}\begin{gathered}\label{est3}
\sup_{0\leq u\leq s}  |X_u|\leq  \left(1+s^H\log^2s\right)\zeta
\end{gathered}\end{equation}
and for $s\in\left[\frac{k}{n}, \frac{k+1}{n}\right]$
\begin{equation}\begin{gathered}\label{est4}
\sup_{\frac{k}{n}\le u\le s}\abs{X_u-X_{k,n}}
\leq \left(\tfrac{1}{n}+\tfrac{1}{n} s^H \log^2 s +n^{-H+\alpha} \right)\zeta.
\end{gathered}\end{equation}

\end{rem}

\section{Bounds for the numerator of the estimator}\label{section3}  Now we are in position to bound both terms in the numerator of the right-hand side of \eqref{estimator}.
At first,  give the bound with probability 1 for the 1st term in the numerator of \eqref{estimator}. All inequalities claimed in Lemma \ref{lem3.1} hold for any $\alpha>0$ starting with some nonrandom number $n(\alpha)$.
\begin{lemma}\label{lem3.1} $(i)$ Let $\theta>0$. Then for any
$m>1 $ there exists such $\zeta \in \mathfrak{Z}$ that
 \[
\abs{\sum_{k=0}^{n^m-1}X_{k,n}
\int_{\frac{k}{n}}^{\frac{k+1}{n}}\left(X_s-X_{k,n}\right)ds}
\le   \zeta^2 {n}^{-1} e^{2\theta n^{m-1}}.
\]$(ii)$ Let $\theta<0$. Then we have two cases.
\begin{itemize}

\item [(a)] Let  $1<m\leq \frac1H$. Then there exists such $\zeta \in \mathfrak{Z}$ that
\[
\abs{\sum_{k=0}^{n^m-1}X_{k,n}
\int_{\frac{k}{n}}^{\frac{k+1}{n}}\left(X_s-X_{k,n}\right)ds}
\le   \zeta^2n^{mH+m-2H-1+\alpha}.
\]

\item [(b)] Let  $m>\frac1H$. Then there exists such $\zeta \in \mathfrak{Z}$ that
\[
\abs{\sum_{k=0}^{n^m-1}X_{k,n}
\int_{\frac{k}{n}}^{\frac{k+1}{n}}\left(X_s-X_{k,n}\right)ds}
\le   \zeta^2n^{2Hm+m-2H-2+\alpha}.
\]
\end{itemize}

\end{lemma}
\begin{proof} $(i)$ It follows immediately from \eqref{est31} that \begin{equation}\begin{gathered}\label{est311}
\abs{X_{k,n}}\leq\sup_{0\leq u\leq \tfrac{k+1}{n}}  |X_u|\leq  \left(e^{\theta \frac{k+1}{n}}+\left(\tfrac{k+1}{n}\right)^H\log^2\left(\tfrac{k+1}{n}\right)\right)\zeta.
\end{gathered}\end{equation}
Now we take into account \eqref{est311},  substitute $\tfrac{k+1}{n}$ instead of $s$ into \eqref{est41} and apply Lemma \ref{l:2.4} to get for any $\alpha>0$ the following relations:
\begin{equation}\begin{gathered}\label{est42}
\abs{\sum_{k=0}^{n^m-1}X_{k,n}
\int_{\frac{k}{n}}^{\frac{k+1}{n}}\left(X_s-X_{k,n}\right)ds}
\le
\zeta^2\frac{1}{n}\sum_{k=0}^{n^m-1}\left(e^{\theta\frac{k+1}{n}}+\left(\tfrac{k+1}{n}\right)^H\log^2 \left(\tfrac{k+1}{n}\right)\right)\\ \times \left(\frac{1}{n}e^{\theta\frac{k+1}{n}}+\frac{1}{n}\left(\tfrac{k+1}{n}\right)^H\log^2 \left(\tfrac{k+1}{n}\right)+n^{-H+\alpha}\log n\right)\\=\zeta^2\left(\frac{1}{n^2}\sum_{k=0}^{n^m-1}e^{2\theta\frac{k+1}{n}}
+\frac{2}{n^2}\sum_{k=0}^{n^m-1}e^{\theta\frac{k+1}{n}}\left(\tfrac{k+1}{n}\right)^H\log^2 \left(\tfrac{k+1}{n}\right)\right.\\
+\frac{1}{n^2}\sum_{k=0}^{n^m-1}\left(\tfrac{k+1}{n}\right)^{2H}\log^4 \left(\tfrac{k+1}{n}\right)+n^{-1-H+\alpha}\log n\\ \times\left.\left(\sum_{k=0}^{n^m-1}e^{\theta\frac{k+1}{n}}+\sum_{k=0}^{n^m-1}\left(\tfrac{k+1}{n}\right)^H\log^2 \tfrac{k+1}{n}\right)\right)\\\leq \zeta^2\biggl(\frac{1}{n}e^{2\theta n^{m-1}}+e^{\theta n^{m-1}}\left(n^{H(m-1)+m-2+\alpha}+n^{-H+\alpha}\right)\\
+n^{m-1-H+(m-1)H+\alpha}\biggr)\end{gathered}\end{equation}
Evidently, the term $\frac{1}{n}e^{2\theta n^{m-1}}$ dominates and other terms are negligible, whence the proof of $(i)$ follows.

$(ii)$ According to \eqref{est3}, now
\begin{equation*}\begin{gathered}
\abs{X_{k,n}}\leq\sup_{0\leq u\leq \frac{k+1}{n}}  |X_u|\leq  \left(1+\left(\tfrac{k+1}{n}\right)^H\log^2\tfrac{k+1}{n}\right)\zeta.
\end{gathered}\end{equation*} Substituting  $\tfrac{k+1}{n}$ instead of $s$ into \eqref{est4}, we get  the following relations:
\begin{equation}\begin{gathered}\label{est5}
\abs{\sum_{k=0}^{n^m-1}X_{k,n}
\int_{\frac{k}{n}}^{\frac{k+1}{n}}\left(X_s-X_{k,n}\right)ds}
\le
\zeta^2\frac{1}{n}\sum_{k=0}^{n^m-1}\left(1+\left(\tfrac{k+1}{n}\right)^H\log^2\tfrac{k+1}{n}\right)\\ \times\left(\tfrac{1} {n}+\tfrac{1}{n} \left(\tfrac{k+1}{n}\right)^H \log^2 \left(\tfrac{k+1}{n}\right)  +n^{-H+\alpha}\right)\\
=\zeta^2\left(n^{m-2}
+\frac{2}{n^2}\sum_{k=0}^{n^m-1} \left(\tfrac{k+1}{n}\right)^H \log^2 \left(\tfrac{k+1}{n}\right)+\frac{1}{n^2}\sum_{k=0}^{n^m-1}\left(\tfrac{k+1}{n}\right)^{2H}\log^4\tfrac{k+1}{n}\right.\\
+\left.n^{m-1-H+\alpha}
+n^{-1-H+\alpha}
\sum_{k=0}^{n^m-1}\left(\tfrac{k+1}{n}\right)^H\log^2\tfrac{k+1}{n}\right).
\end{gathered}\end{equation}
Substituting  the bounds from Lemma \ref{l:2.4}  into the right-hand side of \eqref{est5}, we obtain
\begin{gather*}
\abs{\sum_{k=0}^{n^m-1}X_{k,n}
\int_{\frac{k}{n}}^{\frac{k+1}{n}}\left(X_s-X_{k,n}\right)ds}
\le
\zeta^2\Big(n^{m-2}
+n^{(m-1)H-2+m}\log^2 n\\
+n^{2H(m-1)-2+m}\log^4 n+n^{m-1-H+\alpha}
+n^{(m-2)H-1+m+\alpha}
\log^2n\Big).
\end{gather*}
We take into account that $\log n=o(n^\alpha)$ as $n\rightarrow\infty$, for any $\alpha>0$.  So, it is necessary to compare exponents $m-2$, $(m-1)H-2+m$,  $2H(m-1)-2+m$, $m-1-H$  and $(m-2)H-1+m$. We get  that exponent $2H(m-1)-2+m$ is the biggest under the condition $m>\frac1H$ while exponent $(m-2)H-1+m$  is the biggest under the condition $m\leq \frac1H$ whence the proof of $(ii)$ follows.
\end{proof}

Now we establish moment bounds for the 2nd term in the numerator of the right-hand side of \eqref{estimator}.
In order to do this, we apply well-known Isserlis' formula for calculation of higher moments of Gaussian distribution: let $\{\chi_1, \chi_2, \chi_3, \chi_4\}$ be a Gaussian vector, then $$ \E(\chi_1\chi_2\chi_3\chi_4)=\E(\chi_1\chi_2)\E(\chi_3\chi_4)+\E(\chi_1\chi_3)\E(\chi_2\chi_4)+\E(\chi_1\chi_4)\E(\chi_2\chi_3).$$
Therefore,  we can calculate mathematical expectations
$\E B^H_s B^H_u \Delta B^H_k \Delta B^H_j$
for $k\ne j$ as
\begin{multline}\label{mom_bound1}
\E B^H_s B^H_u \Delta B^H_k \Delta B^H_j
=\E B^H_u \Delta B^H_k \E B^H_s \Delta B^H_j
+\E B^H_s \Delta B^H_k \E B^H_u \Delta B^H_j\\
+\E B^H_s B^H_u \E \Delta B^H_k\Delta B^H_j
\le
\E B^H_u \Delta B^H_k \E B^H_s \Delta B^H_j
\\+\E B^H_s \Delta B^H_k \E B^H_u \Delta B^H_j
\end{multline}
because for
$H\in\left(0,\frac12\right)$
the increments of fBm $B^H$ are negatively correlated and so
$\E\Delta B^H_k \Delta B^H_j<0$.
Similarly,
\begin{equation}\label{mom_bound2}
\E B^H_s B^H_u\left(\Delta B^H_{k,n}\right)^2
=2\E B^H_u \Delta B^H_{k,n} \E B^H_s \Delta B^H_{k,n}
+n^{-2H}\E B^H_s B^H_u.
\end{equation}
\begin{lemma} $(i)$  Let $\theta>0$.  Then for any $m>1$  we have the following moment bound
\[
\E\left(\sum_{k=0}^{n^m-1}X_{k,n}\Delta B^H_{k,n} \right)^2\le Cn^{2-4H}e^{2\theta n^{m-1}}.
\]

$(ii)$  Let $\theta<0$. Then for any $m>1$  we have the following moment bound
\[
\E\left(\sum_{k=0}^{n^m-1}X_{k,n}\Delta B^H_{k,n} \right)^2\le C n^{2m-4H}.
\]
\end{lemma}
\begin{proof} $(i)$
It follows from \eqref{solution} that
\begin{multline*}
\sum_{k=0}^{n^m-1}X_{k,n}\Delta B^H_{k,n}
=x_0\sum_{k=0}^{n^m-1}e^{\theta\frac{k}{n}}\Delta B^H_{k,n}\\
+\theta\sum_{k=0}^{n^m-1}e^{\theta\frac{k}{n}}
\int_0^{\frac{k}{n}}e^{-\theta s}B^H_s\,ds
\cdot\Delta B^H_{k,n}
+\sum_{k=0}^{n^m-1}B^H_{k,n}\Delta B^H_{k,n}
=:I_n^1+I_n^2+I_n^3.
\end{multline*}

Remind that
$\E\Delta B^H_{k,n}\Delta B^H_{j,n}<0$
for $k\ne j$.
Therefore
\begin{multline*}
0\le\E\left(\sum_{k=0}^{n^m-1}e^{\theta\frac{k}{n}}\Delta B^H_{k,n}\right)^2\\
\le n^{-2H}\sum_{k=0}^{n^m-1}e^{2\theta\frac{k}{n}}
\leq  n^{1-2H}\int_0^{n^{m-1}}e^{2\theta s}ds=Cn^{1-2H}e^{2\theta n^{m-1}}.
\end{multline*}
So,
$\E\left(I_n^1\right)^2\le Cn^{1-2H}e^{2\theta n^{m-1}}$, $n\ge1$.

Consider $I_n^3$.
It is well known (see, e.\,g., the relation (1.8) from \cite{Nourdin}) that for
$H\in\left(0,\frac12\right)$
\[
n^{2H-1}\sum_{k=0}^{n-1}B^H_{k,n}\Delta B^H_{k,n}
\xrightarrow{L^2}-\frac12.
\]
Therefore, there exists $C>0$ such that
\[
\E\left(n^{2H-1}\sum_{k=0}^{n-1}B^H_{k,n}\Delta B^H_{k,n}\right)^2\le C,\quad n\ge1.
\]

Now we can use the self-similarity property of $B^H$, namely,
$$\left(B^H_{at},t\ge0\right)\overset{d}{=}
a^H\left(B^H_{t},t\ge0\right)$$
and get
\begin{gather*}
\E\left(I_n^3\right)^2
=\E\left(\sum_{k=0}^{n^m-1}B^H_{k,n}\Delta B^H_{k,n} \right)^2\\
=
n^{4(1-m)H}\E\left(\sum_{k=0}^{n^m-1}
B^H_{\frac{k}{n^m}}
\left(B^H_{\frac{k+1}{ n^m }}
-B^H_{\frac{k}{n^m}}\right)\right)^2\\
\le
n^{4(1-m)H}\cdot n^{2m(1-2H)}
\E\left(n^{m(2H-1)}\sum_{k=0}^{ n^m -1}
B^H_{\frac{k}{ n^m }}
\left(B^H_{\frac{k+1}{ n^m }}
-B^H_{\frac{k}{ n^m }}\right)\right)^2\\
\le
C n^{4H-4m H+2m-4m H}
=C n^{4H+2m-8m H}.
\end{gather*}

At last,
\begin{equation}\label{est61}
\begin{gathered}
0\le\E\left(I_n^2\right)^2
=\theta^2\sum_{j,k=0}^{n^m-1}
e^{\theta\frac{k}{n}+\theta\frac{j}{n}}
\int_0^{\frac{k}{n}}\int_0^{\frac{j}{n}}
e^{-\theta s-\theta u}\,\E B^H_sB^H_u\Delta B^H_{k,n}\Delta B^H_{j,n}\,du\,ds\\
=
\theta^2\sum_{k=0}^{n^m-1}e^{2\theta\frac{k}{n}}
\int_0^{\frac{k}{n}}\int_0^{\frac{k}{n}}
e^{-\theta s-\theta u}\,\E B^H_sB^H_u \left(\Delta B^H_{k,n}\right)^2du\,ds\\
+
\theta^2\sum_{k\ne j}e^{\theta\frac{k}{n}+\theta\frac{j}{n}}
\int_0^{\frac{k}{n}}\int_0^{\frac{j}{n}}
e^{-\theta s-\theta u}\,\E B^H_sB^H_u\Delta B^H_{k,n}\Delta B^H_{j,n}\,du\,ds
=:J_1^n+J_2^n.
\end{gathered}
\end{equation}
We get from  \eqref{mom_bound1} and \eqref{mom_bound2} that
\begin{equation}\begin{gathered}\label{incr1}
\E B^H_sB^H_u \left(\Delta B^H_{k,n}\right)^2
=\frac12\left(\left(\tfrac{k+1}{n}\right)^{2H}-\abs{u-\tfrac{k+1}{n}}^{2H}
-\left(\tfrac{k}{n}\right)^{2H}+\abs{u-\tfrac{k}{n}}^{2H}\right)\\
\times \left(\left(\tfrac{k+1}{n}\right)^{2H}-\abs{s-\tfrac{k+1}{n}}^{2H}
-\left(\tfrac{k}{n}\right)^{2H}+\abs{s-\tfrac{k}{n}}^{2H}\right)+n^{-2H}\E B^H_s B^H_u\\ \leq C n^{-4H}+C n^{-2H}\left(s^{2H}+u^{2H}-|s-u|^{2H}\right),
\end{gathered}\end{equation}
and
\begin{equation}\begin{gathered}\label{incr2}
\E B^H_sB^H_u \Delta B^H_{k} \Delta B^H_{j}
\le\frac12\left(\left(\tfrac{k+1}{n}\right)^{2H}-\left(\tfrac{k}{n}\right)^{2H}
-\abs{u-\tfrac{k+1}{n}}^{2H}
+\abs{u-\tfrac{k}{n}}^{2H}\right)\\
\times \left(\left(\tfrac{k+1}{n}\right)^{2H}-\left(\tfrac{k}{n}\right)^{2H}
-\abs{s-\tfrac{k+1}{n}}^{2H}
+\abs{s-\tfrac{k}{n}}^{2H}\right)\leq C n^{-4H}.
\end{gathered}\end{equation}
Substituting the above bounds into \eqref{est61}, we get that
\[
J_1^n\le Cn^{1-2H}e^{2\theta n^{m-1}}
\quad\text{and}\quad
J_2^n\le Cn^{2-4H}e^{2\theta n^{m-1}}.
\]
Evidently, for $m>1$ the biggest contribution is from the bound  $Cn^{2-4H}e^{2\theta n^{m-1}}$ whence the proof follows.

$(ii)$ Let $\theta<0$. In this case
\begin{multline*}
0\le\E\left(\sum_{k=0}^{n^m-1}e^{\theta\frac{k}{n}}\Delta B^H_{k,n}\right)^2
\le\sum_{k=0}^{n^m-1}e^{2\theta\frac{k}{n}}
\E\left(\Delta B^H_{k,n}\right)^2\\
\le\sum_{k=0}^{n^m-1}\E\left(\Delta B^H_{k,n}\right)^2
=n^m n^{-2H}
=n^{m-2H}.
\end{multline*}
So,
$\E\left(I_n^1\right)^2\le Cn^{m-2H}$, $n\ge1$.

The term  $I_n^3$ is estimated as before, and
\begin{equation}\label{est6}
\begin{gathered}
0\le\E\left(I_n^2\right)^2
=\theta^2\sum_{j,k=0}^{n^m-1}
e^{\theta\frac{k}{n}+\theta\frac{j}{n}}
\int_0^{\frac{k}{n}}\int_0^{\frac{j}{n}}
e^{-\theta s-\theta u}\,\E B^H_sB^H_u\Delta B^H_{k,n}\Delta B^H_{j,n}\,du\,ds\\
=
\theta^2\sum_{k=0}^{n^m-1}e^{2\theta\frac{k}{n}}
\int_0^{\frac{k}{n}}\int_0^{\frac{k}{n}}
e^{-\theta s-\theta u}\,\E B^H_sB^H_u \left(\Delta B^H_{k,n}\right)^2du\,ds\\
+
\theta^2\sum_{k\ne j}e^{\theta\frac{k}{n}+\theta\frac{j}{n}}
\int_0^{\frac{k}{n}}\int_0^{\frac{j}{n}}
e^{-\theta s-\theta u}\,\E B^H_sB^H_u\Delta B^H_{k,n}\Delta B^H_{j,n}\,du\,ds
=:J_1^n+J_2^n.
\end{gathered}
\end{equation}

Substituting   bounds \eqref{incr1} and \eqref{incr2} into \eqref{est6}, we get that
\begin{gather*}
J_1^n\leq Cn^{-4H}\sum_{k=0}^{n^m-1}e^{2\theta\frac{k}{n}}
\int_0^{\frac{k}{n}}\int_0^{\frac{k}{n}}
e^{-\theta s-\theta u}ds\,du\\+Cn^{-2H}\sum_{k=0}^{n^m-1}e^{2\theta\frac{k}{n}}
\int_0^{\frac{k}{n}}\int_0^{\frac{k}{n}}
e^{-\theta s-\theta u}\left(s^{2H}+u^{2H}-|s-u|^{2H}\right)ds\,du
\\
\leq Cn^{2m-4H}+Cn^{-2H}\sum_{k=0}^{n^m-1}e^{\theta\frac{k}{n}}
\int_0^{\frac{k}{n}}
e^{-\theta s} s^{2H} ds\\
\leq Cn^{2m-4H}+Cn^{1-2H}\int_0^{n^{m-1}}e^{\theta s}\int_0^{s}
e^{-\theta u} u^{2H} du\,ds\\
\leq Cn^{2m-4H}+Cn^{1-2H}e^{\theta n^{m-1}}\int_0^{n^{m-1}}e^{-\theta u}u^{2H}du+Cn^{1-2H}\int_0^{n^{m-1}}u^{2H}du\\ \leq Cn^{2m-4H}+Cn^{1-2H}\int_0^{n^{m-1}}u^{2H}du\\ \leq Cn^{2m-4H}+Cn^{2Hm+m-4H}\leq Cn^{2m-4H},
\quad\text{and}\quad
J_2^n\le Cn^{2m-4H}.
\end{gather*}
Comparing exponents
$2m-4H$, $4H+2m-8Hm$ and $m-2H$,
we get that for $m>1$
$2m-4H$ is the biggest one, whence the proof follows.
\end{proof}

\begin{corollary}\label{corol2}
$(i)$ Let $\theta>0$. Then \[
\E\left(n^{4H-2}e^{-2\theta n^{m-1}}\sum_{k=0}^{n^m-1}X_{k,n}\Delta B^H_{k,n} \right)^2\le C.
\]
If we denote $\xi_n=n^{2H-1}e^{- \theta n^{m-1}}\sum_{k=0}^{n^m-1}X_{k,n}\Delta B^H_{k,n}  $ then $\sup_{n\ge1}\E\xi_n^2<\infty$. It means that for any $m>1$ the numerator of \eqref{estimator} can be bounded by the sum
\[\zeta^2 {n}^{-1} e^{2\theta n^{m-1}} +n^{1-2H}e^{ \theta n^{m-1}}\xi_n,
\]
where $\sup_{n\ge1}\E\xi_n^2<\infty$.

$(ii)$ Let $\theta<0$. Then we have two cases.
\begin{itemize} \item[$(a)$]
 Let  $1<m\leq \frac1H$. Then for any $\alpha>0$ the numerator of \eqref{estimator} can be bounded by the sum
\[
\zeta^2n^{(m-2)H+m-1+\alpha}+n^{m-2H}\xi_n,
\]
where $\sup_{n\ge1}\E\xi_n^2<\infty$.
 \item[$(b)$] Let  $m>\frac1H$. Then for any $\alpha>0$ the numerator of \eqref{estimator} can be bounded by the sum
\[
\zeta^2n^{(2H+1)m-2H-2+\alpha}+n^{m-2H}\xi_n,
\]
where $\sup_{n\ge1}\E\xi_n^2<\infty$.
\end{itemize}
\end{corollary}
\section{How to deal with the denominator}\label{section4}

Now our goal is to present the denominator of \eqref{estimator} in more convenient form.
At first we compare the sum
$\frac1n\sum_{k=0}^{n^m-1}X_{k,n}^2$
to the corresponding integral \linebreak
$\int_0^{n^{m-1}}X_s^2\,ds$. The reason to replace the sum with the corresponding integral is that for some values of $H$ and $m$ we can prove the consistency with the help of some kind of  L'H\^{o}pital's rule, however,   the application of L'H\^{o}pital's rule    or   Stolz–-Ces\`{a}ro theorem to the sum  $\sum_{k=0}^{n^m-1}X_{k,n}^2$  is problematic because not only the upper bound but the terms in the sum depend on $n$.

\begin{lemma}\label{l:2.8} $(i)$  Let $\theta>0$. Then there exists such  $\zeta_1\in \mathfrak{Z}$ that
\[
\abs{\int_0^{n^{m-1}}X_s^2\,ds
-\frac1n\sum_{k=0}^{n^m-1}X_{k,n}^2}
\leq  \frac{\zeta_1^2}{n} e^{2\theta n^{m-1}}.
\]

$(ii)$ Let $\theta<0$. Then we have two cases.
\begin{itemize}
\item [$(a)$] Let  $1<m\leq \frac1H$. Then there exists such  $\zeta_1\in \mathfrak{Z}$ that for any $\beta>0$
  we have the following   bound

\[
\abs{\int_0^{n^{m-1}}X_s^2\,ds
-\frac1n\sum_{k=0}^{n^m-1}X_{k,n}^2}
\leq  \zeta_1^2 n^{ mH+m-2H-1+\beta}.
\]
\item [$(b)$] Let  $m> \frac1H$. Then there exists such  $\zeta_1\in \mathfrak{Z}$ that for any $\beta>0$
  we have the following   bound

\[
\abs{\int_0^{n^{m-1}}X_s^2\,ds
-\frac1n\sum_{k=0}^{n^m-1}X_{k,n}^2}
\leq  \zeta_1^2 n^{ 2mH+m-2H-2+\beta}.
\]
\end{itemize}
\end{lemma}
\begin{proof}
Evidently, the difference between the integral and the corresponding integral sum can be bounded  as
\[
\abs{\int_0^{n^{m-1}}X_s^2\,ds
-\frac1n\sum_{k=0}^{n^m-1}X_{k,n}^2}
\le\int_0^{n^{m-1}}|\varphi_n(s)|\,ds,
\]
where the integrand has the form
\[
\varphi_n(s)=\left(X_s^2-X_{k,n}^2\right)
\ind_{s\in\left[\frac{k}{n},\frac{k+1}{n}\right)},
\quad 0\le k\le n^m-1.
\]
Furthermore,  the integrand can be bounded as

\begin{gather*}
\abs{\varphi_n(s)}
\le\abs{X_s-X_{k,n}}\left(\abs{X_s}+\abs{X_{k,n}}\right)
\ind_{s\in\left[\frac{k}{n},\frac{k+1}{n}\right)}\\ \le 2\abs{X_s-X_{k,n}}
\sup_{0\leq u\leq s}\abs{X_u}\ind_{s\in\left[\frac{k}{n},\frac{k+1}{n}\right)}.\end{gather*}
$(i)$ Let $\theta>0$. Then from \eqref{est31}, \eqref{est41} and similarly  to \eqref{est42},
\begin{equation}\begin{gathered}\label{est12}\abs{X_s-X_{k,n}}
\sup_{0\leq u\leq s}\abs{X_u}\ind_{s\in\left[\frac{k}{n},\frac{k+1}{n}\right)}\le \Big(n^{-1}e^{2\theta s}+2n^{-1}e^{\theta s}s^H\log^2s\\+e^{\theta s}n^{-H+\alpha}+n^{-1}s^{2H}\log^4s+n^{-H+r}s^H\log^2s\Big)\zeta_1^2.\end{gathered}\end{equation}
Integrating over $[0, n^{m-1}]$, we see that the integral of the first term in the right-hand side of \eqref{est12} dominates, whence the proof follows.

$(ii)$ Let $\theta<0$. Then according to \eqref{est3}--\eqref{est4}, \begin{gather*}\abs{X_s-X_{k,n}}
\sup_{0\leq u\leq s}\abs{X_u}\ind_{s\in\left[\frac{k}{n},\frac{k+1}{n}\right)}\le
\Big(\tfrac{1} {n}+\tfrac{1}{n} \left(\tfrac{k+1}{n}\right)^H \log^2 \left(\tfrac{k+1}{n}\right)  \\
+n^{-H+\alpha}\Big)
\left(1+\left(\tfrac{k+1}{n}\right)^H
\log^2\left(\tfrac{k+1}{n}\right)\right)\zeta_1^2,
\end{gather*}
therefore
\begin{equation}\label{est7}
\begin{gathered}
\int_0^{n^{m-1}}|\varphi_n(s)|\,ds
\le\left(n^{m-2}+\frac{2}{n^2}\sum_{k=0}^{n^m-1}
\left(\tfrac{k+1}{n}\right)^H\log^2\tfrac{k+1}{n}
\right.\\+\frac{1}{n^2}\sum_{k=0}^{n^m-1}
\left(\tfrac{k+1}{n}\right)^{2H}\log^4\tfrac{k+1}{n}+n^{m-H-1+\alpha}
\\+\left.n^{-H-1+\alpha}\sum_{k=0}^{n^m-1}
\left(\tfrac{k+1}{n}\right)^H\log^2\tfrac{k+1}{n}\right)
\zeta_1^2.
\end{gathered}
\end{equation}
To get rid of logarithms, we apply  Lemma~\ref{l:2.4} to \eqref{est7} and obtain  that for any $\beta>0$
\begin{multline*}
\int_0^{n^{m-1}}|\varphi_n(s)|\,ds
\le\left(n^{m-2}+n^{mH+m-H-2+\beta}
+n^{m-H-1+\beta}\right.\\
+\left.n^{2Hm+m-2H-2+\beta}
+n^{mH+m-2H-1+\beta}\right)
\zeta_1^2.
\end{multline*}
Comparing the exponents $m-2$, $mH+m-H-2$, $m-H-1$,  $2Hm+m-2H-2$ and $mH+m-2H-1$,
we deduce that for $1<m\leq\frac1H$ the biggest exponent equals $mH+m-2H-1$, and for
$m>\frac1H$
the biggest exponent equals $2Hm+m-2H-2$,
 whence the proof follows.
\end{proof}

\begin{corollary}\label{corol3}
 $(i)$  Let $\theta>0$. Then there exists such  $\zeta_1\in \mathfrak{Z}$ that
\[\frac1n\sum_{k=0}^{n^m-1}X_{k,n}^2=\int_0^{n^{m-1}}X_s^2\,ds+\vartheta_n,\]
where \[\abs{\vartheta_n}
\leq  \frac{\zeta_1^2}{n} e^{2\theta n^{m-1}}.
\]

$(ii)$ Let $\theta<0$. Then we have two cases.
\begin{itemize}
\item [$(a)$] Let  $1<m\leq \frac1H$. Then there exists such  $\zeta_1\in \mathfrak{Z}$ that for any $\beta>0$

\begin{equation}\label{repr1}
\frac1n\sum_{k=0}^{n^m-1}X_{k,n}^2=\int_0^{n^{m-1}}X_s^2\,ds+\vartheta_n(\beta),\end{equation}
where \begin{equation}\label{est40}
\abs{\vartheta_n(\beta)}
\leq   \zeta_1^2 n^{ mH+m-2H-1+\beta}.\end{equation}

\item [$(b)$] Let  $m> \frac1H$. Then for any $\beta>0$
 representation \eqref{repr1} holds with
 \[\abs{\vartheta_n(\beta)}
\leq   \zeta_1^2 n^{ 2mH+m-2H-2+\beta}.
\]
\end{itemize}

\end{corollary}

{\section{Main consistency results}\label{section5} }

Consider separately cases $\theta>0$ and $\theta<0$. The case $\theta>0$ is more simple and additionally the estimator is strongly consistent. The case $\theta<0$ needs some additional calculations and we prove only that the estimator is consistent.
\begin{theorem}\label{main1} Let $\theta>0$. Then for any $m>1$ estimator  $\widehat{\theta}_n(m)$ is strongly consistent.\end{theorem}
\begin{proof} According to Corollaries \ref{corol2} and \ref{corol3},  it is sufficient to prove that $$\psi_n:=\frac{\zeta^2 {n}^{-1} e^{2\theta n^{m-1}} +n^{1-2H}e^{ \theta n^{m-1}}\xi_n}{\int_0^{n^{m-1}}X_s^2\,ds+\vartheta_n}\rightarrow  0
$$
a.s. as $n\to\infty$, where $\sup_{n\ge1}\E\xi_n^2<\infty$ and $\abs{\vartheta_n}
\leq  \frac{\zeta_1^2}{n} e^{2\theta n^{m-1}}.$ Rewrite $\psi_n$ as
$$\psi_n:=\frac{\zeta^2 {n}^{-1} +n^{1-2H}e^{-\theta n^{m-1}}\xi_n}{e^{-2\theta n^{m-1}}\int_0^{n^{m-1}}X_s^2\,ds+e^{-2\theta n^{m-1}}\vartheta_n}.$$
Evidently, $\zeta^2 {n}^{-1}\rightarrow  0$ a.s. as $n\to\infty$. Furthermore, for any $x>0$ $$\P\left\{n^{1-2H}e^{-\theta n^{m-1}}\xi_n>x\right\}\leq \frac{\E\xi_n^2}{x^2n^{4H-2}e^{2\theta n^{m-1}}}\leq \frac{C}{x^2 n^{4H-2}e^{2\theta n^{m-1}}},$$
and the series $\sum_{n=1}^{\infty}\frac{1}{n^{4H-2}e^{2\theta n^{m-1}}}$ converges. It means by Borel--Cantelli lemma that $n^{1-2H}e^{-\theta n^{m-1}}\xi_n\to 0$ a.s. as $n\to\infty$. Evidently, $e^{-2\theta n^{m-1}}|\vartheta_n|\leq \frac{\zeta^2_1}{n}\to 0$ a.s. as $n\to\infty$. At last, according to \eqref{solution} and the L'H\^{o}pital's rule, $$\lim_{T\to\infty}\frac{\int_0^{T}X_s^2\,ds}{e^{2\theta T}}=\lim_{T\to\infty}\frac{X_T^2}{2\theta e^{2\theta T}}=(2\theta)^{-1}\left(x_0+\int_0^\infty e^{-\theta s}B_s^H\,ds\right)^2, $$
and the limit random variable is positive a.s. as the square of Gaussian variable, whence  the proof follows.
\end{proof}

In the case when $\theta<0$ we can establish consistency, but not the strong consistency  of \eqref{estimator} as $n\to\infty$. According to Corollaries \ref{corol2} and \ref{corol3}, we need to bound the following random variables: for
$1<m\leq \frac1H$ \[
K_n^1(m, \alpha,\beta):=\frac{n^{mH+m-2H-1+\alpha}}
{\int_0^{n^{m-1}}X_s^2\,ds+\vartheta_n}
\quad\text{and}\quad
K_n^2:=\frac{n^{m-2H}\xi_n}{\frac1n\sum_{k=0}^{n^m-1}X_{k,n}^2}.
\]
and for $m>\frac1H$ \[
\widetilde{K}_n^1(m, \alpha,\beta):=\frac{n^{2Hm+m-2H-2+\alpha}}
{\int_0^{n^{m-1}}X_s^2\,ds+\vartheta_n}
\] and the same  $K_n^2$.

\begin{lemma}\label{lemma est1}  Let $\theta<0$.  \begin{itemize}

\item [$(i)$] For any $1<m<\frac{2H+1}{H+1}<\frac1H$ there exist such $\alpha>0$ and $\beta>0$ that
$$K_n^1(m, \alpha,\beta)\to0$$ a.s.\ as $n\to\infty$.

\item [$(ii)$] For any $\frac{2H+1}{H+1}\leq m\leq\frac1H$ there exist such $\alpha>0$ and $\beta>0$ that
$$K_n^1(m, \alpha,\beta)\to0$$ in probability \ as $n\to\infty$.
\item [$(iii)$] There exist such $\alpha>0$ and $\beta>0$ that for any $ m>\frac1H$
$$\widetilde{K}_n^1(m, \alpha,\beta)\to0$$ in probability \ as $n\to\infty$.
\end{itemize}
\end{lemma}
\begin{proof} $(i)$ Let  $1<m<\frac{2H+1}{H+1}<\frac1H$. Then exponent $mH+m-2H-1$ is negative. Indeed, for $H<\frac12$ we have inequality $\frac{2H+1}{H+1}<\frac1H$.
At first, choose $\alpha>0$ so that $mH+m-2H-1+\alpha<0$ and put $\beta= \alpha $. Then it is sufficient to note  that
$n^{mH+m-2H-1+\alpha}\rightarrow0$ and $\vartheta_n\to 0$ a.s. as $n\to\infty$ while the integral $\int_0^{n^{m-1}}X_s^2\,ds$ is increasing with probability 1 and tends to nonzero with probability 1 random variable as $n\to\infty.$

$(ii)$  Let  $\frac{2H+1}{H+1}\leq m\leq\frac1H$. Then exponent $mH+m-2H-1$ is positive. Choose $\alpha=\beta$. It is sufficient to prove that there exists such $\alpha>0$ that
$$
n^{-mH-m+2H+1-\alpha}\int_0^{n^{m-1}}X_s^2\,ds
+n^{-mH-m+2H+1-\alpha}\vartheta_n(\alpha)\to\infty$$
in probability as   $n\to\infty$.
In view of \eqref{est40} it  is equivalent  to
\[
n^{-mH-m+2H+1-\alpha}\int_0^{n^{m-1}}X_s^2\,ds
\to\infty
\quad\text{in probability\  as } n\to\infty.
\]

To establish this convergence  note that it follows from Cauchy--Schwarz inequality that
\[
n^{-mH-m+2H+1-\alpha}\int_0^{n^{m-1}}X_s^2\,ds
\geq n^{-mH-2m+2H+2-\alpha}\left(\int_0^{n^{m-1}}X_s \,ds\right)^2.
\]
Denote $\gamma=-mH-2m+2H+2-\alpha<0$. Without loss of generality suppose that $x_0>0.$ Note  that $\int_0^{n^{m-1}}X_s \,ds$ is a Gaussian process  with mean $$e_n=\frac{x_0}{\theta} \Big(e^{\theta n^{m-1}}-1\Big)\in(0, -\frac{x_0}{\theta})$$ and variance
\begin{multline}\label{est8}
\sigma_n^2=\int_0^{n^{m-1}}\!\!\int_0^{n^{m-1}}\E X_sX_t \,ds \,dt-e_n^2\\
=\int_0^{n^{m-1}}\!\!\int_0^{n^{m-1}}
\E\left(\theta e^{\theta s}\int_0^se^{-\theta u}B^H_u\,du +B^H_s\right)\\
\times\left(\theta e^{\theta t}\int_0^te^{-\theta z}B^H_z\,dz +B^H_t\right)ds\,dt.
\end{multline}
Since for any $u,z>0$ we have $\E B^H_uB^H_z>0$, variance can be bounded from below by the following value
\begin{equation}\begin{gathered}\label{est9} \sigma_n^2\geq \int_0^{n^{m-1}}\int_0^{n^{m-1}}\E B^H_sB^H_t\,ds \,dt\\=\frac12n^{(m-1)(2H+2)}\int_0^1\int_0^1\left(s^{2H}+t^{2H}-|s-t|^{2H}\right)\,ds \,dt=Cn^{(m-1)(2H+2)}.\end{gathered}
\end{equation}
Note that other terms in \eqref{est8} are of the same order so bound \eqref{est9} is sharp. Now, denoting  $\mathcal{N}(0,1)$ standard Gaussian random variable and \linebreak $\Phi(x)=(2\pi)^{-\frac12}\int_{-\infty}^xe^{-\frac{y^2}{2}}dy$, we can deduce that for any $A>0$ and sufficiently large $n$
\begin{equation}\begin{gathered}\label{est10} \P \left\{n^\gamma\left(\int_0^{n^{m-1}}X_s\,ds\right)^2\leq A^2\right\}=\P \left\{n^{ \frac{\gamma}{2}}\abs{\int_0^{n^{m-1}}X_s\,ds}\leq A\right\}\\=\P \left\{n^{ \frac{\gamma}{2}}\abs{\sigma_n\mathcal{N}(0,1)+e_n}\leq A\right\}=\Phi\left(\frac{A}{\sigma_nn^{ \frac{\gamma}{2}}}-\frac{e_n}{\sigma_n}\right)- \Phi\left(-\frac{A}{\sigma_nn^{ \frac{\gamma}{2}}}-\frac{e_n}{\sigma_n}\right)\\ \leq\abs{\Phi\left(\frac{A}{\sigma_nn^{ \frac{\gamma}{2}}}-\frac{e_n}{\sigma_n}\right)-\frac12}+\abs{\Phi\left(-\frac{A}{\sigma_nn^{ \frac{\gamma}{2}}}-\frac{e_n}{\sigma_n}\right)-\frac12}\leq 2\left(\frac{A}{\sigma_nn^{ \frac{\gamma}{2}}}-\frac{e_n}{\sigma_n}\right)\\ \leq \frac{C}{\sigma_nn^{ \frac{\gamma}{2}}} \leq\frac{C}{n^{\frac{mH}{2}-\frac{\alpha}{2}}}.
\end{gathered}
\end{equation}
Choosing $0<\alpha< {mH} $ we get the proof of $(ii)$.

$(iii)$. For $m>\frac1H$ exponent $2Hm+m-2H-2$ is positive. Therefore, we repeat the proof of $(ii)$ with the same $\sigma_n$ and  with $\widetilde{\gamma}=-2Hm-2m+2H+3-\alpha$ instead of $\gamma$. So, in the inequality similar to \eqref{est9}, we get in the right-hand side the upper bound $$\frac{C}{\sigma_nn^{ \widetilde{\gamma}/2}}\leq\frac{C}{n^{\frac{1}{2}-\frac{\alpha}{2}}}.$$
Choosing $0<\alpha< \frac12 $ we get the proof of $(iii)$.

\end{proof}
\begin{rem} We can prove  more than it was mentioned in $(i)$, namely, to establish that $$\int_0^{n^{m-1}}X_s^2\,ds\rightarrow\infty$$ a.s. as $n\to\infty$ (see Lemma \ref{lemA}  in Section \ref{sectionAp}).
\end{rem}

\begin{lemma}
Let $m>\frac{1}{2H}$.
Then $K_n^2\to0$    in probability as  $n\to\infty$.
\end{lemma}

\begin{proof} We apply the same method as in the proof of Lemma \ref{lemma est1}, but to the sum $\frac1n
\sum_{k=0}^{n^m-1}X_{k,n}^2$ instead of integral $\int_0^{n^{m-1}}X_s^2\,ds$. As before,  suppose that $x_0>0.$ According to Cauchy--Schwarz inequality,
\begin{gather*}\frac1n
\sum_{k=0}^{n^m-1}X_{k,n}^2 \geq n^{-m-1} \left(\sum_{k=0}^{n^m-1}X_{k,n}\right)^2,
\end{gather*}
where $\sum_{k=0}^{n^m-1}X_{k,n} $ is the Gaussian random variable with mean $$0<\widetilde{e}_n=x_0\sum_{k=0}^{n^m-1}e^{\theta\frac{k}{n}}\leq n x_0\int_0^{n^{m-1}}e^{\theta s}\,ds\leq -\frac{n x_0}{\theta} $$ and variance \begin{gather*}\widetilde{\sigma}_n^2=\sum_{k,j=0}^{n^m-1}\E X_{k,n}X_{j,n}\geq \sum_{k,j=0}^{n^m-1}\E B^H_{k,n}B^H_{j,n}\\
=\frac12 \sum_{k,j=0}^{n^m-1}\left(\left(\tfrac{k}{n}\right)^{2H}
+\left(\tfrac{j}{n}\right)^{2H}-\abs{\tfrac{j}{n}-\tfrac{k}{n}}^{2H}\right)\\
=n^m\sum_{k=0}^{n^m-1}\left(\tfrac{k}{n}\right)^{2H}-\sum_{l=0}^{n^m-1}\left(\tfrac{l}{n}\right)^{2H}(n^m-l)
\\=n^{-2H}\sum_{l=0}^{n^m-1}l^{2H+1}=n^{-2H+m(2H+1)+m}\sum_{l=0}^{n^m-1} \left({\tfrac{l}{n^{m}}}\right)^{2H+1}\tfrac{1}{n^m}
\\\geq Cn^{-2H+m(2H+1)+m}.
\end{gather*} Therefore,
for any $\varepsilon>0$ and $x_n>0$
\begin{gather*}
\Prob\left\{K_n^2\geq\varepsilon\right\}
\leq\Prob\left\{\frac{\abs{\xi_n}n^{m-2H}}{\frac1n
\sum_{k=0}^{n^m-1}X_{k,n}^2}\geq \varepsilon\right\}
\le
\Prob\left\{\frac1n\sum_{k=0}^{n^m-1}X_{k,n}^2<x_n\right\}
\\+\Prob\left\{\abs{\xi_n}>x_n\varepsilon n^{2H-m}\right\}
\le
\Prob\left\{\frac{1}{n^{m+1}}\left(\sum_{k=0}^{n^m-1}X_{k,n}\right)^2<x_n\right\}
+\frac{\E\xi_n^2}{x_n^2\varepsilon^2n^{4H-2m}}\\ \leq \Prob\left\{\frac{1}{n^{m+1}}\left(\sum_{k=0}^{n^m-1}X_{k,n}\right)^2<x_n\right\}
+\frac{C}{x_n^2\varepsilon^2n^{4H-2m}}.
\end{gather*}
Similarly to  \eqref{est10}, \begin{gather*} \Prob\left\{\frac{1}{n^{m+1}}\left(\sum_{k=0}^{n^m-1}X_{k,n}\right)^2<x_n\right\}
=\Prob\left\{\abs{\widetilde{\sigma}_n\mathcal{N}(0,1)+\widetilde{e}_n}<x_n^{\frac12}n^{\frac{m+1}{2}}\right\}
\\=\Phi\left(-\frac{\widetilde{e}_n}{\widetilde{\sigma}_n}+\frac{x_n^{\frac12}n^{\frac{m+1}{2}}}{\widetilde{\sigma}_n}\right)-
\Phi\left(-\frac{\widetilde{e}_n}{\widetilde{\sigma}_n}-\frac{x_n^{\frac12}n^{\frac{m+1}{2}}}{\widetilde{\sigma}_n}\right)
 \leq C\left(\frac{x_n^{\frac12}n^{\frac{m+1}{2}}}{\widetilde{\sigma}_n}+\frac{n}{\widetilde{\sigma}_n}\right).
\end{gather*}

Evidently, for any $m>1$, $\dfrac{n}{\widetilde{\sigma}_n}\to 0$ as $n\to\infty$. Therefore,  to supply convergence $K_n^2\to 0$ in probability, we need to choose   $x_n$ in such a way that $\dfrac{1}{x_n^2n^{4H-2m}}\to0$ and $\dfrac{x_n^{\frac12}n^{\frac{m+1}{2}}}{\widetilde{\sigma}_n}\to0$ as $n\to\infty$. Put $x_n=n^r$. Then $r$ must satisfy the  double inequality $$m-2H<r<m+2Hm-2H-1.$$
This inequality can be satisfied only for $m>\frac{1}{2H}$ whence the proof follows.
\end{proof}

\begin{theorem}\label{main2}
Let $m>\frac{1}{2H}.$ Then the estimator $\widehat{\theta}_n(m)$, introduced in \eqref{estimator}, is consistent.
\end{theorem}

\section{Auxiliary results}\label{sectionAp}

At first we establish an auxiliary result concerning the bounds for several   sums of integral type that will participate in the bounds for the numerator of \eqref{estimator}
{\begin{lemma}\label{l:2.4}
For any $m> 1$ and $n\geq 2$ there exists $C>0$ not depending on $n$ such that
\begin{enumerate}[(i)]
\item\label{l2.4-i}
$$\sum_{k=0}^{n^m-1}
\left(\tfrac{k+1}{n}\right)^H
 \log^2\tfrac{k+1}{n}\le
C n^{(m-1)H+m}\log^2 n,
$$
\item\label{l2.4-ii}
$$\sum_{k=0}^{n^m-1}
\left(\tfrac{k+1}{n}\right)^{2H} \log^4\tfrac{k+1}{n}
\le C n^{2H(m-1)+m}\log^4 n.
$$
\end{enumerate}
\end{lemma}}
\begin{proof}
We base the proof of both statements on the following evident relation: for any function $f:[0,1]\rightarrow  \R$ that is  Riemann integrable on $[0,1]$, and for any $m\geq 1$ integral sums $S(f(x), n^m):=\frac{1}{n^m}\sum_{k=0}^{n^m-1}f \left(\frac{k+1}{n^m}\right)$ tend  to integral $\int_0^1f(x)dx$ as $n\rightarrow\infty$. In particular,  these integral sums are bounded. Consider the statement \eqref{l2.4-i}. Evidently, \begin{equation*}\begin{gathered}
\sum_{k=0}^{n^m-1}
\left(\tfrac{k+1}{n}\right)^H\log^2\tfrac{k+1}{n}
=n^{(m-1)H}\sum_{k=0}^{n^m-1}
\left(\tfrac{k+1}{n^m}\right)^H
\log^2(\tfrac{k+1}{n^m}n^{m-1})\\
\le 2n^{(m-1)H}\sum_{k=0}^{n^m-1}
\left(\tfrac{k+1}{n^m}\right)^H\left(\log^2\left(\tfrac{k+1}{n^m}\right)+(m-1)^2\log^2n\right)\\
=2n^{(m-1)H+m}S(x^H\log^2x, n^m)
+2(m-1)^2n^{(m-1)H+m}\log^2n\cdot S\left(x^H,n^m\right)\\
\leq Cn^{(m-1)H+m}+ Cn^{(m-1)H+m}\log^2n
\le Cn^{(m-1)H+m}\log^2n\end{gathered}\end{equation*}
for   $n\geq 2$. Statement \eqref{l2.4-ii} is established similarly.

\end{proof}

Next auxiliary result establishes asymptotic behavior of integral $\int_0^T X_s^2\,ds$ as $T\to\infty.$

\begin{lemma}\label{lemA} Let process $X$ satisfy equation \eqref{equation}. Then
$\int_0^T X_s^2\,ds\to\infty$ with probability 1 as $T\to\infty$.
\end{lemma}
\begin{proof} The result is obvious for $\theta>0$ therefore we consider only the case $\theta<0.$ Since $\int_0^T X_s^2\,ds$ is nondecreasing in $T>0$, it is sufficient to prove that $\int_0^T X_s^2\,ds\to\infty$ in probability.
For any $\lambda>0$ consider the moment generation function   $\Theta_T(\lambda)=\E \exp\left\{-\lambda \int_0^T X_s^2\,ds\right\}$ and $\Theta_\infty(\lambda)=\E \exp\left\{-\lambda \int_0^\infty X_s^2\,ds\right\}$  so that  $$\Theta_\infty(\lambda)=\lim_{T\rightarrow\infty}\Theta_T(\lambda).$$  Evidently,
    \begin{eqnarray} \begin{gathered}
    \int_0^T X_s^2\,ds\geq T^{-1} \left(\int_0^TX_s\,ds\right)^2,\nonumber
    \end{gathered}
    \end{eqnarray}
    whence $$\Theta_T(\lambda)\leq \Theta_T^{(1)}(\lambda):=\E\exp\left\{-\frac{ \lambda}{T}\left(\int_0^TX_s\,ds\right)^2\right\}.$$
     Random  variable $T^{-\frac12}\int_0^TX_s\,ds$ is Gaussian with  mean $m(T)$ and variance $\sigma^2(T)$, say.
     Note that for a Gaussian random variable $\xi=m+\sigma \mathcal{N}(0,1)$ we have that  \begin{equation*}
     \E\exp\left\{- \lambda\xi^2\right\}=\left(2\lambda\sigma^2+1\right)^{-\frac12}\exp\left\{-\frac{\lambda m^2}{2a\sigma^2+1}\right\}\leq \left(2\lambda\sigma^2+1\right)^{-\frac12}.
     \end{equation*}
      Therefore,   it is sufficient to prove that $$\lim_{T\rightarrow\infty}\sigma^2(T)=\infty.$$
    Similarly to \eqref{est9}.
    \begin{equation*}\begin{gathered}\sigma^2(T)\geq T^{2H+1}\int_0^1\int_0^1\Big(s^{2H}+t^{2H}-|s-t|^{2H}\Big)\,ds\,dt\to\infty
    \end{gathered}
    \end{equation*}
    as $T\to\infty$, whence the proof follows.
\end{proof}

\section{Simulations}\label{section7}

In this section, we present the results of simulation experiments.
We simulate 20 trajectories of the fractional Ornstein--Uhlenbeck process \eqref{equation} with $x_0=1$ for different values of $\theta$ and $H$.
Then we compute values of $\widehat\theta_n(m)$.
For each combination of $\theta$, $H$, $n$ and $m$ the mean of the estimator is reported.

In Tables \ref{tab1}--\ref{tab3} the true value of the drift parameter $\theta$ equals 2.
In this case the behavior of the estimators is almost the same for different values of $H$.
Also we can see that the value of $\widehat\theta_n(m)$ is determined by $n$ and does not depend on $m$.
Further, we consider the case of negative $\theta$.
We simulate the process with $H=0.45$, $\theta=-3$ and $m=4,5$.
The results are reported in Tables \ref{tab4}--\ref{tab5}. One  can see that the method works but the rate of convergence to the true value of a parameter is not very high. There are two reasons for this: the estimator is only consistent not strongly consistent and moreover, the trajectories are so irregular that even the length of the interval is small we can not ``catch'' the trajectory.

\begin{table}[h!]
\caption{\label{tab1}$\theta=2$, $m=2$.}
\vspace*{-5mm}
\small
\hspace{1.5cm}
\begin{center}
\begin{tabular}{|c|c|c|c|c|c|c|}\hline
$n$ & $5$ & $10$ & $50$ & $100$  & $500$ & $1000$\\
\hline
$H=0.05$ & $2.45763$ & $2.21281$ & $2.0395$ & $2.01911$ & $2.00300$ & $2.00100$
\\\hline
$H=0.25$ & $2.45766$ & $2.21281$ & $2.0395$ & $2.01911$ & $2.00300$ & $2.00100$
\\\hline
$H=0.45$ & $2.45794$ & $2.21281$ & $2.0395$ & $2.01911$ & $2.00300$ & $2.00100$
\\\hline
\end{tabular}
\end{center}
\end{table}

\begin{table}[h!]
\caption{\label{tab2}$\theta=2$, $m=3$.}
\vspace*{-5mm}
\small
\hspace{1.5cm}
\begin{center}
\begin{tabular}{|c|c|c|c|c|}\hline
$n$ & $5$ & $10$ & $20$ & $25$ \\
\hline
$H=0.05$ & $2.45763$ & $2.21281$ & $2.10231$ & $2.08109$
\\\hline
$H=0.25$ & $2.45763$ & $2.21281$ & $2.10231$ & $2.08109$
\\\hline
$H=0.45$ & $2.45763$ & $2.21281$ & $2.10231$ & $2.08109$
\\\hline
\end{tabular}
\end{center}
\end{table}

\begin{table}[h!]
\caption{\label{tab3}$\theta=2$, $m=4$.}
\vspace*{-5mm}
\small
\hspace{1.5cm}
\begin{center}
\begin{tabular}{|c|c|c|c|c|c|}\hline
$n$ & $5$ & $8$ & $10$ & $12$ & $15$\\
\hline
$H=0.05$ & $2.45763$ & $2.27092$ & $2.21281$ & $2.17240$ & $2.13566$
\\\hline
$H=0.25$ & $2.45763$ & $2.27092$ & $2.21281$ & $2.17240$ & $2.13566$
\\\hline
$H=0.45$ & $2.45763$ & $2.27092$ & $2.21281$ & $2.17240$ & $2.13566$
\\\hline
\end{tabular}
\end{center}
\end{table}

\begin{table}[h!]
\caption{\label{tab4}$\theta=-3$, $H=0.45$, $m=4$.}
\vspace*{-5mm}
\small
\hspace{1.5cm}
\begin{center}
\begin{tabular}{|c|c|c|c|c|c|}\hline
$n$ & $2$ & $4$ & $6$ & $8$ & $10$
\\\hline
$\vphantom{\Big|}\widehat\theta_n(4)$ & $-1.50913$ & $-2.41157$ & $-2.71411$ & $-2.9546$ & $-3.12058$
\\\hline
\end{tabular}
\end{center}
\end{table}

\begin{table}[h!]
\caption{\label{tab5}$\theta=-3$, $H=0.45$, $m=5$.}
\vspace*{-5mm}
\small
\hspace{1.5cm}
\begin{center}
\begin{tabular}{|c|c|c|c|c|c|}\hline
$n$ & $2$ & $3$ & $4$ & $5$ & $6$
\\\hline
$\vphantom{\Big|}\widehat\theta_n(5)$ & $-1.63396$ & $-2.04297$ & $-2.38237$ & $-2.5595$ & $-2.72538$
\\\hline
\end{tabular}
\end{center}
\end{table}


\end{document}